\documentclass[12pt]{amsart}
\usepackage{amssymb,amsmath,amsthm}

\newtheorem{theorem}{Theorem}[section]
 \newtheorem{lemma}[theorem]{Lemma}

\newtheorem{Label}[theorem]{}

\usepackage{epsfig}

\begin{document}
\begin{center}

{\large\bf
 The first coefficient of Homflypt and Kauffman polynomials:
        Vertigan proof of polynomial complexity using dynamic programming}
\end{center}
 \begin{center}
by J\'ozef H.Przytycki
 \end{center}
 \vspace{0.1in}
\begin{quotation}\ \\
{\bf Abstract.}
We describe the polynomial time complexity algorithm
for computing  first coefficients of the skein (Homflypt)
and Kauffman polynomial invariants of links, discovered by D.Vertigan
in 1992 but never published.
\end{quotation}

\section{Introduction}\label{Section 1}
We showed in \cite{P-P-2} that an essential part of the Jones-type
polynomial link invariants can be computed in subexponential time. 
This is in a sharp contrast to the result of Jaeger, 
Vertigan and Welsh \cite{JVW} that computing the whole
polynomial and most of its evaluations is $\#P$-hard
and is conjectured to be of exponential complexity.

Motivated by \cite{P-P-2}, Dirk Vertigan described
the polynomial time complexity algorithm
for computing  first coefficients of the skein (Homflypt) and Kauffman
polynomials of links\footnote{On 13 Jan 1992 we got an e-mail
from Paul Seymour, editor of Proceedings to which \cite{P-P-2}
was submitted informing as that:
``The referee for your paper on polynomials for the Seattle meeting has done
some further work of his own, extending the results in your paper, and now
he is worried that he has abused his position as referee for his own gain.
I asked him to summarize his results and send them to me, and told him I would
pass them on to you. So please, what are your reactions? Do you have any
objections to the referee publishing the stuff below as his own work?"
We were very enthusiastic about the referee's result but he somehow
never published the paper, and we included his description in
the appendix of our preprint \cite{P-P-1}.}. The polynomial time
complexity of other coefficients follows easily from the first
coefficient. We express the time complexity of our algorithms as 
a function of the number of crossings, $n$, and we assume 
that the number of link components, com$(L)$, of a link $L$ 
is less than or equal to the number of crossings.

The skein (Homflypt) polynomial, $P_L(a,z) \in Z[a^{\pm 1},z^{\pm 1}]$,
 of oriented links in $R^3$ is 
defined recursively as follows \cite{HOMFLY,PT}:
\begin{description}
\item [(i)]
$P_{trivial knot}(a,z) = 1$,
\item [(ii)]
$aP_{L_+}(a,z) + a^{-1}P_{L_-}(a,z) = zP_{L_0}(a,z).$
\end{description}
Let $com(L)$ denote the number of components of $L$ then 
$z^{com(L)-1}P_L(a,z) \in Z[a^{\pm 1},z]$ and it 
can be written as 
$\Sigma_{i=0}^M P_{2i}(a)z^{2i}$ \cite{L-M}. 

\begin{theorem}[Vertigan]\ \\
  $P_{2i}(a)$ can be computed in polynomial 
time. More precisely: let $D$ be a diagram of $L$ with $n$ 
crossings then the time complexity of computing 
$P_{2i}(a)(L)$ is $O(n^{2+3i})$.

\end{theorem}
In fact Vertigan announced $O(n^{2+2i})$ time algorithm but the proof is
more involved than that of Theorem 1.1, in which case one easily reduces 
the theorem for $P_{2i}(a)$ for links  to the result for
$P_0(a) $ for knots. We describe the case of $P_0(a)$ first.

\section{Computation of $P_0(a)$}
\begin{theorem}\label{2.1}
Let $n$ denote the number of crossings of a knot diagram then
$P_0(a)$ can be computed in quadratic time (i.e. in $O(n^2)$ time).
\end{theorem}
\begin{proof} Let $D$ be an oriented knot diagram with $n$ crossings.
We can think of $D$ as a 4-regular graph (any crossing is a vertex
of valency four). Choose a point inside any edge of $D$ and order them
according to the orientation of the knot. So we get points $b_0,b_1,...,
b_{m-1},b_m=b_0$ where $m$ is the number of edges of $D$ (in fact
$m = 2n$). We think of $b_0$ 
as a base point of $D$. Let $D_{i,j},(0\leq i<j\leq m),$ denote
the part of $D$ between points $b_i$ and $b_j$ (with the convention that
$D_{0,m}$ denotes $D$). 
Further let $\hat D_{i,j}$ denote the closure of $D_{i,j}$,
that is we join, in $\hat D_{i,j}$,  
$b_j$ with $b_i$ by an overpass (an arc going above the rest of the diagram), 
compare Fig.2.1.

\ \\
\centerline{\psfig{figure=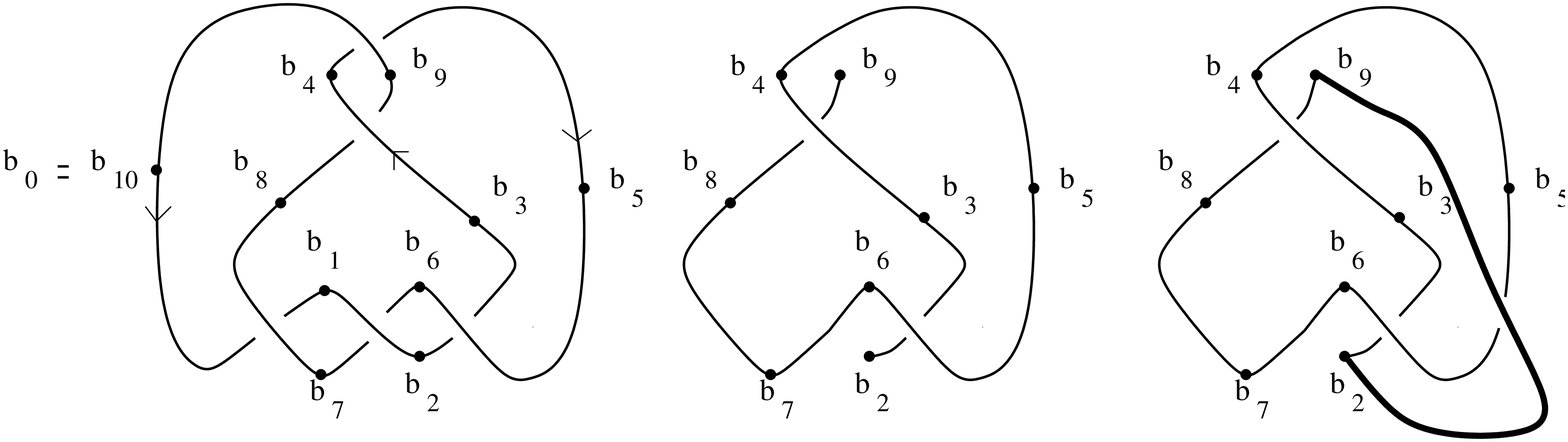,height=4.6cm}}
\centerline{Fig. 2.1;\ \ 
$D=D_{0,10}=\hat D_{0,10}$,\ $D_{2,9}$\  and \ $\hat D_{2,9}$}

\begin{lemma}\label{x.3}
If $j-i \leq 3$ then $\hat D_{i,j}$ represents the unknot.
\end{lemma}
\begin{proof}
$D_{i,j}$, for $j-i\leq 3,$ can have at most one 
crossing and $\hat D_{i,j}$ can be drawn with
no more than one crossing. Therefore $\hat D_{i,j}$ represents the unknot.
\end{proof}
Notice that $\hat D_{i,i+4}$ cannot represent a nontrivial knot 
neither.
$\hat D_{i,i+5}$ can represent only a trefoil knot or the unknot
(compare Section 5).

To continue the proof of Theorem 2.1 first observe that 
if diagrams $D_+, D_-$ and $D_0$ form a skein triplet 
then the skein relation for the skein
polynomial $P(a,z)$ reduces, for $P_0(a)$, to the formula:
\[ aP_0(a)(D_+) + a^{-1}P_0(a)(D_-) = \left \{
\begin{array}{ll}
 P_0(a)(D_0) & \mbox{in the case of a selfcrossing}\\
0 & \mbox{in the case of a crossing}\\
 & \mbox{between different components}
\end{array}
\right .
\]

For the trivial link of $n$ components, $T_n$, one has 
$P_0(a)(T_n)= (a+a^{-1})^{n-1}$. Now consider $D_{i,j}$ which has a crossing 
(otherwise $\hat D_{i,j}$ represents the unknot). Let $q$ be the first crossing
in $D_{i,j}$ after $b_i$. Without lost of generality we can assume that the
arc $b_i, b_{i+1}$ is involved in the crossing 
(otherwise $D_{i,j}=D_{i+1,j}$).
We have two possibilities:
\begin{enumerate}
\item[(i)] the arc $b_i,b_{i+1}$ is an overpass and then $\hat D_{i,j} = \hat D_{i+1,j}$, or
\item[(ii)]
 the arc $b_i,b_{i+1}$ is an underpass and in that case we consider the skein triplet
   $\hat D_{i,j}, \hat D'_{i,j}$ and $\hat D^0_{i,j}$ where the second
   element of the triplet is obtained from the first by changing at $q$ the
undercrossing to the overcrossing and the third by smoothing it at $q$.
The important observation here is that $\hat D'_{i,j} = \hat D_{i+1,j}$ and
that $\hat D^0_{i,j}$ is a two component link composed of $\hat D_{i+1,k}$
and $\hat D_{k+1,j}$ where $i<k\leq j$ and $q$ is the crossing between
arcs $b_i,b_{i+1}$ and $b_k,b_{k+1}$ (compare Fig.2.2). 
The first coefficient of a two
component link can be easily computed from that of the components (see \cite
{L-M} or formula 3.1). Therefore we get:
$$a^{\epsilon (q)}P_0(a)(\hat D_{i,j})+
a^{-\epsilon (q)}P_0(a)(\hat D_{i+1,j})= $$
$$ (-a^{-2})^{lk(\hat D^0_{i,j})}(a+a^{-1})P_0(a)(\hat D_{i+1,k})
P_0(a)(\hat D_{k+1,j})$$
where $\epsilon (q)$ is the sign of the crossing $q$ and $lk(L)$ the 
global linking number of the link $L$.
\end{enumerate}

(i) and (ii) allow as to reduce the computation of $P_0(a)(\hat D_{i,j})$
to that of $P_0(a)(\hat D_{s,t})$ with $i<s$ and $t \leq j$. Furthermore we
know the value of $P_0(a)(\hat D_{i,j})$  
to be equal to $1$ for $j-i\leq 3$ by Lemma 2.2.
Therefore we can find the value of $P_0(a)(\hat D_{i,j})$ for any $0\leq i
<j \leq m$ , including that for $D=D_{0,m}=\hat D_{0,m}$, in at most
$m^2 /2 =2n^2$ steps. This completes the proof of Theorem 2.1.
\end{proof}

\ \\
\centerline{\psfig{figure=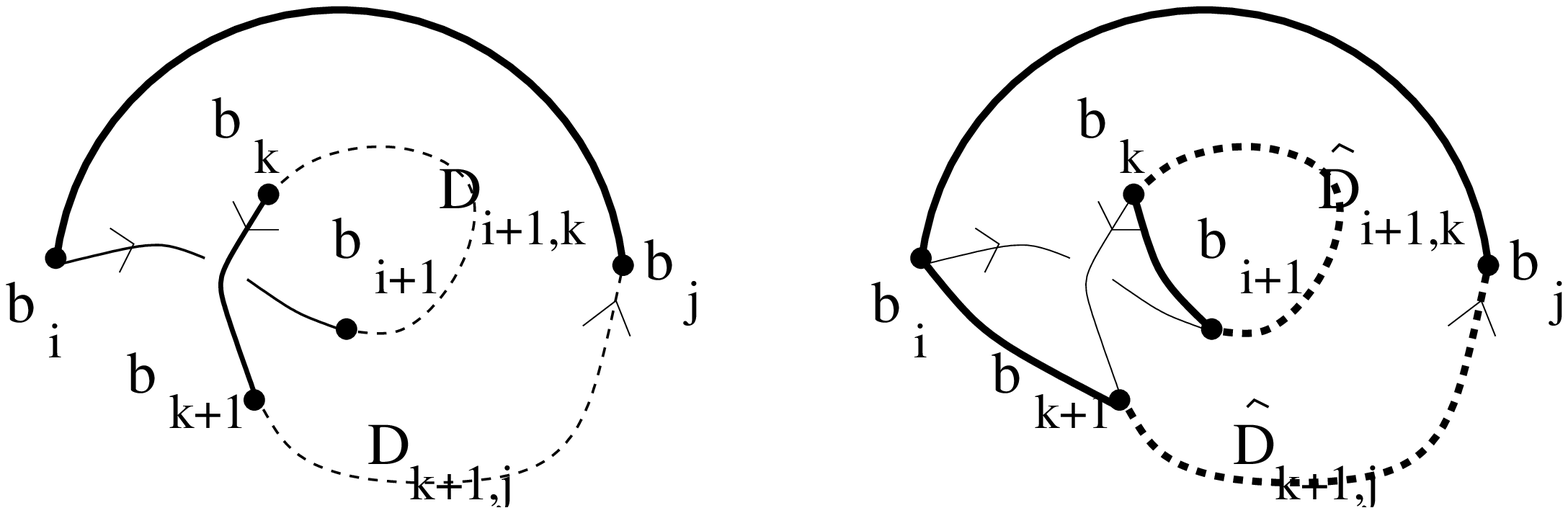,height=4.7cm}}
\begin{center}
        Fig. 2.2;\ \  $\hat D_{i,j}$, $\hat D_{k+1,j}$ and $\hat D_{i+1,j}$
\end{center}

Note that we do not address technical details of complexity of presenting
the computed polynomial in the ordered form. One can improve constant
by considering $D$ or its mirror image $\bar D$ and observing that
$D$ or $\bar D$ can be changed to a descending diagram by switching
no more than $\frac{n}{2}$ crossings.

\section{Computation of $P_{2i}(a)$.}

To finish the proof of Theorem 1.1 first observe that 
Theorem 2.1 can
be extended to the case of a link by the Lickorish-Millett formula 
\cite {L-M}:

\begin{Label}\label{3.1} \ \\

For a link $L$ of $com(L)$ components $K_1,K_2,...K_{com(L)}$
$$P_0(a)(L)=(-a^{-2})^{lk(L)}(a+a^{-1})^{com(L)-1}
\Pi_{i=1}^{com(L)}P_0(a)(K_i)$$
\end{Label}

We assume that the number of components of a link is not too big 
with respect to the number of crossings. 
It remains to see that one can find $P_{2i+2}(a)$ in $O(n^{2+3(i+1)})$ time
assuming that $P_{2i}(a)$ can be found in $O(n^{2+3i})$ time.
We use the generalization of Formula 3.1 to any coefficient $P_{2i}(a)$:

\begin{Label}\label{3.2}

$$P_{2i+2}(a)(L)=(-a^{-2})^{lk(L)}(a+a^{-1})^{com(L)-1}
\Pi_{j=1}^{com(L)}P_{2i+2}(a)(K_j) + \Sigma_{j=1}^{n'}P_{2i}(a)(L_j)$$ 
where $n'$ denotes the number of crossings between different 
components of the considered digram of $L$ 
(therefore $n' \leq n$)and $L_j$'s are certain $n-1$
crossing $com(L)-1$ component link diagrams obtained from $L$.
\end{Label}
Formula 3.2 follow from the recursive relation:
\[ aP_{2n+2}(a)(D_+) + a^{-1}P_{2n+2}(a)(D_-) = 
\left \{
\begin{array}{ll}
 P_{2n+2}(a)(D_0) & \mbox{in the case of a selfcrossing}\\
P_{2n}(a)(D_0) & \mbox{in the case of a crossing}\\
 & \mbox{between different components}
\end{array}
\right .
\]

Then we proceed exactly as in the proof of Theorem 2.1 except that for the
value of $P_{2i+2}(a)(\hat D^0_{i,j})$ one has to use 
formula 3.2 instead of 3.1.

\section{Coefficients of the Kauffman polynomial, $F_L(a,z)$.}

The Vertigan algorithm can be used also to compute first coefficients
of the Kauffman polynomial, $F_L(a,z)$, in polynomial time. One can write
$z^{com(L)-1}F_L(a,z)$ as $\Sigma_{i=0}^NF_i(a)z^i$.

\begin{theorem}[Vertigan] \label{1.6} 
$F_i(a)$ can be computed in polynomial time. More precisely: let $D$ be
a diagram of $L$ with $n$ crossings, then the time complexity of computing
$F_i(a)(L)$ is $O(n^{2+2i})$.
\end{theorem}
\begin{proof}(sketch)
The main point of the proof is the observation that $F_0(a)(L)=P_0(a)(L)$
(compare \cite{Pr} or \cite{Li})). The additional information needed
in the proof is the skein relation connecting coefficients of the Kauffman
polynomial of diagrams $D_+, D_-, D_0$ and $D_{\infty}$: 

\begin{Label}\label{1.7}

$$a^{w(D_+)}aF_{i+2}(a)(D_+)+a^{w(D_-)}a^{-1}F_{i+2}(a)(D_-) =$$
\[ = \left \{ \begin{array}{ll}
 a^{w(D_0)}F_{i+2}(a)(D_0) +a^{w(D_{\infty})}F_{i+1}(a)(D_{\infty}) &
 \mbox {in the case of a selfcrossing} \\
a^{w(D_0)}F_i(a)(D_0) +a^{w(D_{\infty})}F_i(a)(D_{\infty})
& \mbox{ in the case of a mixed crossing}
\end{array}
\right .
\]
 where $D_+, D_-$ and $D_0$ are consistently oriented diagrams.
For $D_{\infty}$ we can choose any orientation which agrees with that of $D_+$
outside components involved in the crossing. $w(D)$ is the planar writhe (or
Tait number) of $D$ equal to the algebraic sum of signs of crossings.
\end{Label}
\end{proof}

\section{Polynomials of virtual diagrams.}
As a comment to the note after Lemma 2.2 one should stress that $D_{i,i+4}$
from Figure 5.1 cannot be obtained from any diagram $D$, so formally
if $j-i =4$ then $\hat D_{i,i+4}$ represents the unknot. Only
$\hat D_{i,i+5}$ can represent a trefoil (as illustrated in Figure 5.2).

\ \\
\centerline{\psfig{figure=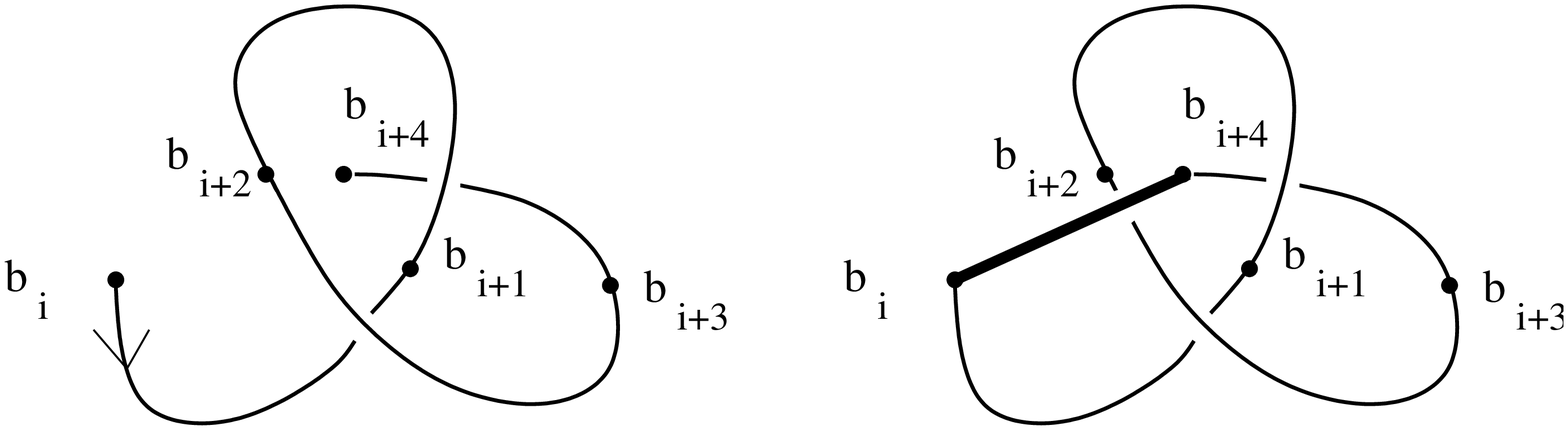,height=4.7cm}}
\begin{center}
        Fig. 5.1;\ \  $D_{i,+4}$,\ \ \ \ \ \ $\hat D_{i,i+4}$ 
\end{center}

\ \\
\centerline{\psfig{figure=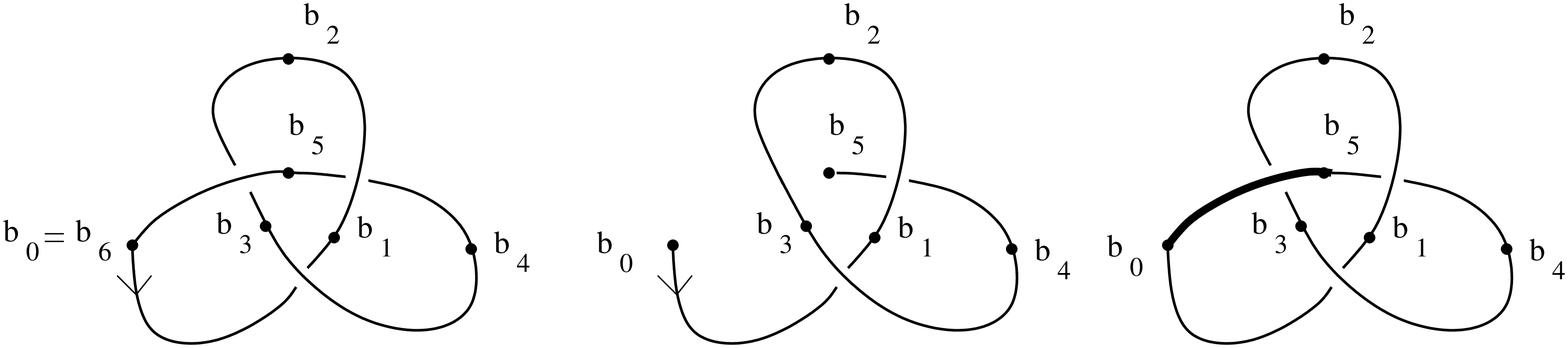,height=4.3cm}}
\begin{center}
      Fig. 5.2;\ \  $D = D_{0,6}$,\ \ \ $D_{0,5}$ \ \ and \ \ $\hat D_{0,5}$
\end{center}

   However, more possibilities arrive if we allow virtual 
diagrams (as introduced by Kauffman \cite{Kau}). 
It may be interesting to use Vertigan algorithm 
for skein (Homflypt) and Kauffman polynomials of virtual knots.

\section{Dynamic programming}\label{Section 6}

The method of dynamic programming, used in Vertigan algorithm is not familiar in knot theory circles, thus we give a short, 
historically based, introduction to the topic.

From \cite{CLR}:\\
R.Bellman began the systematic study of dynamic programming in 1955.
The word ``programming," both here and in linear programming,
refers to the use of a tabular solution method. Although
optimization techniques incorporating elements of dynamic
programming were known earlier, Bellman provided the area with a solid
mathematical basis (Richard Bellman \cite{Be}).

Dynamic programming is effective when a given subproblem
may arise from more than one partial set of choices; the key technique is to
store, or ``memorize," the solution to each such subproblem 
in case it should reappear.
...this simple idea can easily transform exponential-time algorithms into
polynomial-time algorithms.

Example: Longest common subsequence.\\
$O(mn)$-time algorithm for the longest-common-subsequence 
problem seems to be a folk algorithm.\\
In a longest-common-subsequence problem, we are given two sequences
$X=(x_1,x_2,...,x_m)$ and $Y=(y_1,y_2,...,y_n)$ 
and wish to find a maximum-length
common subsequence of $X$ and $Y$."

Another example of dynamic programming is used in H.~Morton's
algorithm computing the Homflypt polynomial of
closed $k$ braids (fixed $k$) in polynomial time with respect to the
number of crossings \cite{M-S}.

\section{Knotoids of Vladimir Turaev}\label{Section 7}
One should mention here that the theory of Knotoids introduced by V. Turaev in 2010 \cite{Tur} is, at least in its pictographic form, very much related 
to Vertigan approach to compute first coefficients of the Homflypt and Kauffman polynomials.

\section{Acknowledgements}\label{Section 8}
I would like to thank Dirk Vertigan for permission to use his result in this note.

I was partially supported by the Simons Foundation Collaboration Grant for Mathematicians--316446.

\ \\
Department of Mathematics,\\
The George Washington University,\\
Washington, DC 20052\\
e-mail: {\tt przytyck@gwu.edu},\\
and University of Gda\'nsk, Poland

\end{document}